\documentclass[11pt]{amsart}
\usepackage{verbatim}
\usepackage{amsmath}
\usepackage{enumerate}
\usepackage{amssymb}
\usepackage{color}
\usepackage{url,longtable,mathrsfs}

\newcommand\T{\rule{0pt}{3.0ex}}       
\newcommand\B{\rule[-1.5ex]{0pt}{0pt}} 

\renewcommand{\phi}{\varphi}
\newcommand\+{\;\lower\plusheight\hbox{$+$}\;}
\newcommand\lldots{\;\lower\plusheight\hbox{$\cdots$}\;}

\newcommand{\mathsym}[1]{{}}

\newcommand{\mat}[4] {\left(\begin{array}{cc} #1 & #2 \\ #3 &
    #4 \end{array}\right)}

\newcommand{\bbZ}[0]  { \mathbb{Z}}

\newtheorem{Theorem}{Theorem}[section]
\newtheorem{Lemma}[Theorem]{Lemma}
\newtheorem{Corollary}[Theorem]{Corollary}
\newtheorem{Definition}[Theorem]{Definition}

\newtheorem{Algorithm}{Algorithm}[section]

\def\Im{\mathop{\rm Im}\nolimits}

\setbox0=\hbox{$+$}
\newdimen\plusheight
\plusheight=\ht 0 \setbox0=\hbox{$-$}
\newdimen\minusheight
\minusheight=\ht0

\setbox0=\hbox{$\cdots$}
\newdimen\cdotsheight
\cdotsheight=\plusheight

\begin{document}


\title[Level 20]{Series for $1/\pi$ of
  signature $20$}
\author{Tim Huber, Dan Schultz and Dongxi Ye}
\address{
School of Mathematical and Statistical Sciences, University of Texas Rio Grande
Valley, Edinburg, Texas 78539, USA.}
\email{timothy.huber@utrgv.edu}
\address{
Department of Mathematics, Pennsylvania State University, University Park, State
College, Pennsylvania 16802, USA.}
\email{dps23@psu.edu}
\address{
Department of Mathematics, University of Wisconsin\\
480 Lincoln Drive, Madison, Wisconsin 53706, USA.}
\email{lawrencefrommath@gmail.com}


\begin{abstract}
Properties of theta functions and Eisenstein series
dating to Jacobi and Ramanujan are used
to deduce differential equations associated with McKay Thompson
series of level $20$. These equations induce expansions for
modular forms of level $20$ in terms of modular functions. The theory
of singular values is applied to derive expansions for $1/\pi$ of
signature $20$ analogous to those formulated by Ramanujan. 
\end{abstract}
\maketitle


\numberwithin{equation}{section}
\allowdisplaybreaks

\section{Introduction}

Let $|q| < 1 $ and define the three null theta functions by 
\begin{align} \label{jac0}
\theta_{3}(q) = \sum_{n= - \infty}^{\infty} q^{n^{2}}, \qquad \theta_{4}(q) = \sum_{n = - \infty}^{\infty} (-1)^{n} q^{n^{2}}, \qquad \theta_{2}(q) = \sum_{n=-\infty}^{\infty} q^{(n + 1/2)^{2}}.
\end{align}
In the Nineteenth Century, Jacobi proved that each $\vartheta =
\theta_{i}(q)$ satisfies a third order differential equation
\cite{MR1578608}
and that the triple of null theta functions satisfies a quartic relation, namely,
  \begin{align} \label{jac}
    \theta_{3}^{4}(q) = \theta_{4}^{4}(q) + \theta_{2}^{4}(q).
  \end{align}
Jacobi referred to \eqref{jac} as an ``\textit{aequatio identica satis
  abstrusa}'' \cite[p. 147]{funnov}. 
Ramanujan proved an equivalent coupled system of differential equations for Eisenstein series on the
full modular group.
This system and Ramanujan's parameterizations for Eisenstein series in
terms of theta functions will be used to formulate a coupled system
for level $2$ modular forms 
\begin{align}
\label{deq1} q\frac{d}{dq}\theta_{3}^{4}
  &=\frac{1}{3}\left(\theta_{2}^{8}-\theta_{4}^{8}+\theta_{3}^{4}P_{2}\right),
  \quad 
q\frac{d}{dq}\theta_{4}^{4}
  =\frac{1}{3}\left(\theta_{2}^{8}-\theta_{3}^{8}+\theta_{4}^{4}P_{2}\right), \\
q\frac{d}{dq}\theta_{2}^{4}
  &=\frac{1}{3}\left(\theta_{3}^{8}-\theta_{4}^{8}+\theta_{2}^{4}P_{2}\right),
  \quad 
q\frac{d}{dq}P_{2}
=\frac{2P_{2}^{2}-\theta_{2}^{8}-\theta_{3}^{8}-\theta_{4}^{8}}{12}, \label{deq2}
\end{align}
where $P_n= P(q^n)$ is defined in terms of the weight two Eisenstein series 
$$
P(q)=1-24\sum_{j=1}^{\infty}\frac{jq^{j}}{1-q^{j}}.
$$
We apply \eqref{deq1}--\eqref{deq2} to obtain a third order differential
equation for $\theta_{3}^4$ with rational coefficients in the level $4$
modular function $\theta_{2}^{4}\theta_{4}^{4}/\theta_{3}^{8}$. From
these identities and \eqref{jac}, we derive a third order differential
equation for certain level 20 modular forms with
coefficients in the field of rational functions generated by McKay Thompson series
of level $20$. The theory of singular values 
is then used to construct
new Ramanujan-Sato expansions for $1/\pi$.


The McKay-Thompson series for the
subgroups $\Gamma$ that will be relevant here are given in
terms of the Dedekind eta-function, defined for $q = e^{2 \pi i\tau}$
and $\tau\in\mathbb{H}$ by
$\eta(\tau)=q^{1/24}\prod_{j=1}^{\infty}(1-q^{j})$ with $\eta_\ell=
\eta(\ell\tau)$. These are  \cite{MR554399} 

\begin{align}
\label{uvdef}
u&=\left(\frac{\eta_{1}\eta_{20}}{\eta_{4}\eta_{5}}\right)^{2},\quad v=\left(\frac{\eta_{2}\eta_{5}\eta_{20}}{\eta_{1}\eta_{4}\eta_{10}}\right)^{2}\\
\label{kwdef}
k&=
\left(\frac{\eta_{4}\eta_{20}}{\eta_{2}\eta_{10}}\right)^{2},\quad w =
\left(\frac{\eta_{2}\eta_{20}}{\eta_{4}\eta_{10}}\right)^{3}.
\end{align}

In the following table, 
we list the appropriate group and Hauptmodul using the notation of \cite{MR554399} for $\Gamma_{0}(20)$ extended by
the indicated Atkin-Lehner involutions. The function $t_{\Gamma}$ represents
the normalized Hauptmodul for $\Gamma$.
\begin{longtable}{|c|c|}
\hline $\Gamma$ & $t_{\Gamma}$\T\\ \hline
 \text{20+} &  $\frac{1}{u}-2+u$ \T\B\\ \hline
 \text{20$|$2+} & $\frac{1}{w} - w$ \T\B\\ \hline
 \text{20+4} & $\frac{1}{v} +2$\T\B\\ \hline
 \text{20$|$2+5} & $\frac{1}{k}$ \T\B\\ \hline
 \text{20$|$2+10} & $\frac{1}{w}$ \T\B\\ \hline
 \text{20+20} & $\frac{1}{u}-2$
                \T \B \\ 
                \hline
\end{longtable}
As in the lower level analogues, weight two forms that appear in
our analysis are theta functions corresponding to binary quadratic
forms. The class number $h(-20) = 2$, and the corresponding inequivalent forms are 
\begin{align*}
  Z=\left(\sum_{m=-\infty}^{\infty}\sum_{n=-\infty}^{\infty}q^{m^{2}+5n^{2}}\right)^{2},\quad
  \mathbf{Z} = \left ( \sum_{m=-\infty}^{\infty}
\sum_{n=-\infty}^{\infty}  q^{2m^{2} + 2m n+ 3n^{2}} \right )^{2}.
\end{align*}
Although we focus on differential equations satisfied by $Z$, sister
equations for $\mathbf{Z}$ may be derived using a similar
construction to that appearing here.

We require the notion of an eta-product, a function of the form
\begin{equation}
\label{etapr}
f(\tau)=\prod_{\delta | \ell} \left(\eta(\delta \tau)\right)^{r_\delta}
\end{equation}
where $\ell$ is a positive integer, the product is taken over the positive
divisors of $\ell$, and the $r_\delta$ are integers. Let $M_k(\Gamma_0(\ell))$ be
the space of modular forms of weight~$k$ with trivial multiplier system
for the modular subgroup $\Gamma_0(\ell)$;
see, e.g., \cite[Chapter 1]{MR2020489} for the definitions.
When $k$ is an even integer there is a simple test that can be used to determine
if an eta-product is in~$M_k(\Gamma_0(\ell))$:
\begin{Lemma}
\label{honda}
Let $\ell$ be a positive integer and consider the eta-product $f(\tau)$ defined by \eqref{etapr}.
Let
$$
k=\frac12\sum_{\delta | \ell} r_{\delta}\quad\mbox{and}\quad s=\prod_{\delta | \ell} \delta^{\left| r_{\delta}\right|}.
$$
Suppose that{\em
\begin{enumerate}
\item $k$ is an even integer;
\item $s$ is the square of an integer;
\item $\displaystyle{\sum_{\delta | \ell} \delta\,r_{\delta} \equiv 0 \pmod{24}}$;
\item $\displaystyle{\sum_{\delta | \ell}\frac{\ell}{ \delta}\,r_{\delta} \equiv 0 \pmod{24}}$;
\item $\displaystyle{\sum_{\delta | \ell} \mbox{gcd}(d,\delta)^2\,\frac{r_\delta}{ \delta} \geq 0}$ for all $d|\ell$.
 \end{enumerate}}
Then $f \in M_k(\Gamma_0(\ell))$.
\end{Lemma}
\begin{proof}
This is immediate from \cite[Thms. 1.64, 1.65]{MR2020489}. The main ideas of
the proof are given in \cite[Theorem 1]{honda}.
\end{proof}
We will need the following result about Eisenstein series of weight 2.
\begin{Lemma}
\label{sch}
Define $P_{\ell} = P(q^{\ell})$. For any positive integer $\ell \geq 2$, $$\ell P_{\ell}-P_{1}\in M_2(\Gamma_0(\ell)).$$
\end{Lemma}
\begin{proof}
See \cite[pp. 177--178]{schoeneberg}.
\end{proof}
The collection of results in the next theorem will be necessary to
develop spanning sets for the spaces of modular forms in the remainder
of the paper.
\begin{Theorem}\label{th14}
The dimension of the space of modular forms of weight $2$ for the modular
subgroup $\Gamma_0(20)$ is given by
\begin{equation}
\label{dim14}
\dim M_2(\Gamma_0(20))=6.
\end{equation}
If $c_1$, $c_2$, $c_4$, $c_{5}$, $c_{10}$ and $c_{20}$ are any constants that satisfy
$$
20c_{1}+10c_{2}+5c_{4}+4c_{5}+2c_{10}+c_{20}=0,
$$
then
\begin{equation}
\label{P12714}
c_{1}P_{1}+c_2P_2+c_4P_4+c_{5}P_{5}+c_{10}P_{10}+c_{20}P_{20}\in M_2(\Gamma_0(20)).
\end{equation}
Furthermore,
\begin{equation}
\label{in14}
z,\,zu,\,zu^{-1},\,zv,\,zv^{-1}\in M_2(\Gamma_0(20)).
\end{equation}
\end{Theorem}
\begin{proof}
The dimension formula \eqref{dim14} follows from \cite[Prop. 6.1]{stein}.
The result \eqref{P12714} follows from Lemma \ref{sch} and the trivial property that
$$
M_k(\Gamma_0(\ell)) \subseteq M_k(\Gamma_0(m))\quad\mbox{if}\quad \ell |m.
$$
The results in \eqref{in14} are immediate consequences of Lemma \ref{honda}.
\end{proof}

\section{Spanning Sets}
In this section, we develop bases for weight two and four forms for
$\Gamma_{0}(20)$. It is well-known, e.g., \cite[p. 83]{stein}, that
\begin{equation}
\label{MSE}
M_{k}(\Gamma_{0}(\ell))=E_{k}(\Gamma_{0}(\ell))\oplus S_{k}(\Gamma_{0}(\ell))
\end{equation}
where $E_{k}(\Gamma_{0}(\ell))$ and $S_{k}(\Gamma_{0}(\ell))$ are the subspaces
of Eisenstein series and cusp forms, respectively, of weight $k$ for $\Gamma_0(\ell)$.
From the dimension formulas in \cite[p. 93]{stein}
we find $\dim{E_{2}(\Gamma_{0}(20))}=5$ and $\dim{S_{2}(\Gamma_{0}(20))}=1$.
In fact,
\begin{align}
\label{E20}
E_{2}(\Gamma_{0}(20))&=
                       \mbox{span}_{\mathbb{C}}\left\{2P(q^{2})-P(q),\,4P(q^4)-P(q),\,5P(q^5)-P(q),\atop10P(q^{10})-P(q),\,20P(q^{20})-P(q)\right\} \\
&=\left\{\left.c_{1}P_{1}+c_2P_2+c_4P_4+c_{5}P_{5}\atop+c_{10}P_{10}+c_{20}P_{20}\; \right| \;
{20c_{1}+10c_{2}+5c_{4}+4c_{5}\atop+2c_{10}+c_{20}=0}\right\}
\nonumber
\end{align}
and \cite{MR2058644}
\begin{equation}
\label{S20}
S_{2}(\Gamma_{0}(20)) = \mathbb{C}z,\qquad\text{where}\ z=\eta_{2}^{2}\eta_{10}^{2}.
\end{equation}
Therefore,
\begin{equation}
\label{M20}
M_{2}(\Gamma_{0}(20))=\mbox{span}_{\mathbb{C}}\left\{2P(q^{2})-P(q),\,4P(q^4)-P(q),\,5P(q^5)-P(q),\atop10P(q^{10})-P(q),\,20P(q^{20})-P(q),\,z\right\}.
\end{equation}
Similarly, we compute an explicit basis for $M_{4}(\Gamma_{0}(20))$:
\begin{equation}
\label{M420}
M_{4}(\Gamma_{0}(20))=\mbox{span}_{\mathbb{C}}\left\{Q(q),\,Q(q^2),\,Q(q^4),\,Q(q^5),\,Q(q^{10}),\,Q(q^{20}),\,z^{2},\atop\frac{z^{2}}{u},\,z^{2}v,\,z^{2}k^{2},\,z^{2}kw\right\}.
\end{equation}

The above construction can be applied to give a representation for
each of $zu,\,zu^{-1},\,zv$ and $zv^{-1}$ as the sum of Eisenstein
series and a cusp form.

\begin{Theorem} Let $z$ be defined by \eqref{S20}. Then 
\label{7ec}
\begin{align}
zu&=\frac{1}{72}\left(P_{1}-6P_{2}+20P_{4}-25P_{5}+30P_{10}-20P_{20}\right)+\frac{1}{3}z,\\
\frac{z}{u}&=\frac{1}{72}\left(-5P_{1}+6P_{2}-4P_{4}+5P_{5}-30P_{10}+100P_{20}\right)+\frac{1}{3}z,\\
zv&=\frac{1}{72}\left(-P_{1}+4P_{4}+P_{5}-4P_{20}\right)-\frac{1}{3}z,\\
\frac{z}{v}&=\frac{1}{72}\left(P_{1}-4P_{4}-25P_{5}+100P_{20}\right)-\frac{5}{3}z.
\end{align}
\end{Theorem}
\begin{proof}
These are immediate from \eqref{in14} and \eqref{M20}.
\end{proof}

These relations induce identities between the level $20$ Hauptmodul.

\begin{Theorem}
\label{uvrecip}
Let $u,v$ be defined by \eqref{uvdef}. Then the following identity holds.
\begin{equation}
\frac{1}{u}+u=\frac{1}{v}+4+5v.
\end{equation}
\end{Theorem}
\begin{proof}
By \eqref{dim14} and \eqref{in14} we have that $z,\,zu,\,zu^{-1},\,zv,\,zv^{-1}$ are linearly dependent;
in fact, from Theorem \ref{7ec} we have
$$
4z+5zv+\frac{z}{v}-zu-\frac{z}{u}=0.
$$
The claimed identity may be obtained by dividing both sides by $z$.
\end{proof}

The preceding proof is a prototype for the derivation of a plethora of relations between
the other Hauptmodul of level $20$. For our purposes, we next derive theta quotient
representations for certain linear functions in $v$.
\begin{Theorem}
\label{vvu}
\begin{align}
\label{1v}
1+v=\frac{\eta_{10}^{8}}{\eta_{1}\eta_{4}\eta_{5}^{3}\eta_{20}^{3}}, \qquad
1+5v=\frac{\eta_{2}^{10}\eta_{5}\eta_{20}}{\eta_{1}^{5}\eta_{4}^{5}\eta_{10}^{2}}.
\end{align}
\end{Theorem}

\begin{proof}
By the definition \eqref{S20} of $z$ and Lemma \ref{honda}, we can check that 
$$
z\frac{\eta_{10}^{8}}{\eta_{1}\eta_{4}\eta_{5}^{3}\eta_{20}^{3}}=\frac{\eta_{2}^{2}\eta_{10}^{10}}{\eta_{20}^{3}\eta_{5}^{3}\eta_{4}\eta_{1}}\in M_{2}(\Gamma_{0}(20)).
$$
By \eqref{M20}, it can be shown that
$$
z\frac{\eta_{10}^{8}}{\eta_{1}\eta_{4}\eta_{5}^{3}\eta_{20}^{3}}=-\frac{1}{72}P(q)+\frac{1}{18}P(q^{4})+\frac{1}{72}P(q^{5})-\frac{1}{18}P(q^{20})+\frac{2}{3}z.
$$
In addition, by Theorem \ref{7ec}, we find that
$$
z+zv=-\frac{1}{72}P(q)+\frac{1}{18}P(q^{4})+\frac{1}{72}P(q^{5})-\frac{1}{18}P(q^{20})+\frac{2}{3}z.
$$
Therefore, 
$$
z+zv=z\frac{\eta_{10}^{8}}{\eta_{1}\eta_{4}\eta_{5}^{3}\eta_{20}^{3}},
$$
and dividing both sides by $z$ gives the first identity of
\eqref{1v}. The second identity can be deduced in a similar way. We omit the details.
\end{proof}

%
%

\section{Differential Equations}

In this section, we prove the differential system
\eqref{deq1}--\eqref{deq2} and deduce third order differential
equations from the computations in the prior section. 
\begin{Theorem}
  The differential sytem \eqref{deq1}--\eqref{deq2} holds.
\end{Theorem}
\begin{proof}
Ramanujan's parameterization 
\cite[Equation 5.4.5]{spiritraman} 
\begin{align}
  \label{eq:10}
P(q^{2}) = (1 - 2x) \theta_{3}^{4} + 6x(1 - x) 
\frac{d\theta_{3}^{2}}{dx},\qquad \quad x =
\frac{\theta_{2}^{4}}{\theta_{3}^{4}}
\end{align}
coupled with Jacobi's relation \eqref{jac} and  \cite{spiritraman}   
\begin{align}
  \label{eq:11}
  q \frac{d x}{dq} = \theta_{3}^{4}x(1-x)
\end{align}
 together imply the claimed expression for $q d\theta_{3}^{4}/dq$. The
 differential equation for $\theta_{4}^4$ follows from that of
 $\theta_{3}^{4}$ since the map $q \mapsto
-q$ fixes $\theta_{2}, P_2$ and interchanges $\theta_{3}$ and
$\theta_{4}$. 
We derive the differential equation for $\theta_{2}^{4}$ directly from
those for $\theta_{3}^{4}, \theta_{4}^{4}$ and
Jacobi's relation \eqref{jac}. The expression for $q dP_2/dq$ is equivalent to Ramanujan's differential equation for the Eisenstein
series of weight two on the full modular group. 
\begin{align}
  \label{eq:9}
  q \frac{dP}{dq} = \frac{P^{2} - Q}{12},\qquad Q(q^{2}) = 1 + 240
  \sum_{n=1}^{\infty} \frac{n^{3} q^{2n}}{1 - q^{2n}} = \frac{\theta_{2}^{8} +
  \theta_{3}^{8} + \theta_{4}^{8}}{2}. 
\end{align}
Ramanujan derived the parameterization for $Q(q^{2})$ \cite[Equation 5.4.8] {spiritraman}.
\end{proof}


\begin{Theorem}  \label{Xth} Let $X$ 
be defined by
$$
X=\frac{z}{Z}.
$$
Then 
\label{Xuv}
\begin{align}
\label{Xu}
X&=\frac{u}{(1+u)^{2}}=\frac{v}{(1+v)(1+5v)}.
\end{align}
\end{Theorem}

\begin{proof}
Since $Z\in M_{2}(\Gamma_{0}(20))$, then by \eqref{M20}, it is easy to show that 
$$
Z=-\frac{1}{18}P(q)+\frac{2}{9}P(q^{4})-\frac{5}{18}P(q^{5})+\frac{10}{9}P(q^{20})+\frac{8}{3}z.
$$
By Theorem \ref{7ec}, it follows that
$$
zu+\frac{z}{u}+2z=\frac{\eta_{1}^{24}\eta_{4}^{24}}{\eta_{2}^{48}}.
$$
Therefore
$$
Z=zu+\frac{z}{u}+2z=z\frac{(1+u)^{2}}{u},
$$
which implies the first equality of \eqref{Xu}. The second equality follows from
Theorem~\ref{uvrecip}. 
\end{proof}

\begin{Theorem}
\begin{align}
\label{du}
q\frac{d}{dq}\log{u}&=\frac{z}{u}\sqrt{1-8u-2u^{2}-8u^{3}+u^{4}},\\
\label{dX}
q\frac{d}{dq}\log{X}&=Z\sqrt{(1-4X)(1-12X+16X^{2})}.
\end{align}
\end{Theorem}

\begin{proof}
It is easy to verify that
$$q\frac{d}{dq}\log{u}=\frac{1}{12}P(q)-\frac{1}{3}P(q^{4})-\frac{5}{12}P(q^{5})+\frac{5}{3}P(q^{20}).$$
By Theorem \ref{7ec}, we have
$$
q\frac{d}{dq}\log{u}=-5zv+\frac{z}{v}.
$$
Then together with Theorem \ref{uvrecip}, the above identity implies \eqref{du}. 
By Theorem \ref{Xuv}, we have
$$
X=\frac{u}{(1+u)^{2}} \quad\mbox{and}\quad Z=z\frac{(1+u)^{2}}{u}.
$$
The differentiation formula \eqref{dX} follows from \eqref{du} via
changes of variables. 
\end{proof}

\begin{Theorem} 
\begin{align}
\label{ZXeq}
X^{2}(1-&4X)(1-12X+16X^{2})\frac{d^{3}Z}{dX^{3}}+3X(1-24X+128X^{2}-160X^{3})\frac{d^{2}Z}{dX^{2}}
\nonumber \\
&+(1-56X+464X^{2}-784X^{3})\frac{dZ}{dX}-4(1-20X+54X^{2})Z=0.
\end{align}
\end{Theorem}

\begin{proof}
Let $F$ and $T$ be defined by
$$
F= \theta_{3}^{4}\quad\mbox{and}\quad T=\frac{t}{(1+16t)^{2}}=\frac{\theta_{2}^{4}\theta_{4}^{4}}{16\theta_{3}^{8}}.
$$
From the differential system \eqref{deq1}--\eqref{deq2} and Jacobi's quartic identity, 

\begin{align}
  \label{eq:7}
 q \frac{dT}{dq}=\frac{\theta_{3}^{4}}{16}\frac{\theta_{2}^{4}\theta_{4}^{4}}{\theta_{3}^{8}}\left(1-2\frac{\theta_{2}^{4}}{\theta_{3}^{4}}\right)=\frac{\theta_{2}^{4}\theta_{4}^{4}}{16\theta_{3}^{4}}\left(1-2\frac{\theta_{2}^{4}}{\theta_{3}^{4}}\right). 
\end{align}
By the chain rule, 
\begin{align}
  \label{eq:8}
\frac{dF}{dT}=\left.\frac{dF}{dq}\right/\frac{dT}{dq}=\frac{16\theta_{3}^{8}\left(\theta_{2}^{8}-\theta_{4}^{8}-\theta_{3}^{4}P_{2}\right)}{3\theta_{2}^{4}\theta_{4}^{4}(\theta_{3}^{4}-2\theta_{2}^{4})}.  
\end{align}

Likewise, $d^{2}F/dT^{2},d^{3}F/dT^{3}$ may be expressed by way of the
chain rule as rational functions of
$\theta_{2}^{4},\theta_{3}^{4},\theta_{4}^{4},P_{2}$. By applying
Jacobi's quartic identity \eqref{jac}, we derive 
\begin{equation}
\label{FTeq}
T^{2}(1-64T)\frac{d^{3}F}{dT^{3}}+3T(1-96T)\frac{d^{2}F}{dT^{2}}+(1-208T)\frac{dF}{dT}=8F.
\end{equation}
Theorem \ref{7ec} and the Jacobi triple product identity \cite{spiritraman} for
the theta functions imply
$$
F=z\frac{(1+5v)^{2}}{v}\quad\mbox{and}\quad T=\frac{v}{(1+v)(1+5v)^{5}}.
$$
Making use of these changes of variables and chain rule, the differential
equation \eqref{FTeq} implies that $z$ satisfies a third order
differential equation with respect to $v$. This implies \eqref{ZXeq}
by the change of variables from Theorem \ref{Xuv}
$$
Z=z\frac{(1+v)(1+5v)}{v}\quad\mbox{and}\quad X=\frac{v}{(1+v)(1+5v)}.$$

\end{proof}

Writing $Z_{X} = X \frac{d}{dX}$, the last theorem takes
a more compact form. 
\begin{Theorem} \label{deqc}
  \begin{align*}
  4 x \left(54 x^2-20 x+1\right)Z&+16 x \left(27 x^2-13 
    x+1\right)Z_{X}+24 x (2 x-1) (6 x-1)Z_{XX} \\ &+(4 x-1) \left(16 
    x^2-12 x+1\right)Z_{XXX} = 0.     
  \end{align*}
\end{Theorem}

Theorem \ref{deqc} induces a power series expansion for $Z$ in
terms of $X$.

\begin{Corollary} For $|X| < \frac{1}{8}
  \left(3-\sqrt{5}\right)$,  $$Z = \sum_{n=0}^{\infty} a_{n} X^{n},$$
  where $a_{0} = 1, a_{1} = 4, a_{2} = 20$ and
  \begin{align*}
 a_{n+1} = 4 \frac{ (2 n+1) \left(2 n^2+2 n+1\right)}{(n+1)^3} a_n -16 \frac{n \left(4 n^2+1\right)}{(n+1)^3} a_{n-1}+8 \frac{(2 n-1)^3}{(n+1)^3} a_{n-2} .
     \end{align*}

\end{Corollary}

\section{Singular Values and Ramanujan-Sato Series of Level 20}

In this section, we given an abbreviated proof of the expansions for
$1/\pi$ resulting from the the differential equations for $Z$ and singular values
for the modular function $X(\tau)$. Well known results from the theory
of singular values and properties of the $j$-invariant are used. Most of the results needed may be
found in \cite{MR1400423,MR554399, MR890960}

For any natural number $N$ and $e||N$, consider the set of Atkin-Lehner involutions
\begin{equation*}
W_e = \left\{ \mat{e a}{b}{N c}{e d}\Big|
\begin{array}{l}
(a,b,c,d) \in \bbZ^4\\
e a d - \tfrac{N}{e} b c=1
\end{array} \right\}\text{.}
\end{equation*}
Each $W_e$ is a coset of $\Gamma_0(N)$ with the multiplication rule
\begin{equation*}
W_e W_f \equiv W_{e f/\gcd(e,f)^2} \bmod \Gamma_0(N)\text{.}
\end{equation*}
For any set of indices $e$ closed under this rule, the group
$\Gamma=\cup_e W_e$ is a subgroup of the normalizer in $PSL(2,\Bbb R)$ of
$\Gamma_0(N)$. Such a group is denoted as
$\Gamma_0(N)+w_{e_1},w_{e_2}\cdots$, or more succinctly as
$N+e_1,e_2,\cdots$. This is shortened to $N+$ when all of the indices
are present. For each such $\Gamma$ of genus $0$, let $t_\Gamma(\tau)$ be its normalized Hauptmodul
\begin{equation*}
    t_\Gamma(\tau) = \frac{1}{q} + \sum_{n=1}^{\infty}{a_n q^n}\text{.}
\end{equation*}

The modular function $X(\tau)$ defined in Theorem \ref{Xth} is the
normalized Hauptmodul for $\Gamma_{0}(20)+$ and is therefore invariant
under action by any element of $\Gamma_{0}(20)+$. To find explicit
evaluations for $X(\tau)$ in the upper half plane $\Bbb H$, we 
construct modular equations of level $20$ and degree $n\ge 2$. 

\begin{Definition}Let $t_\Gamma$ be a normalized Hauptmodul for
  $\Gamma$ and suppose that $\gcd(n,N) = 1$. Then the modular equation
  for $t_\Gamma$ is 
  \begin{align}
    \label{eq:14}
     \Psi_{n}^{t_{\Gamma}}(Y)=  \prod_{
    \underset{\underset{(\alpha,\beta,\delta)=1}{0\le\beta<\delta}}{\alpha
    \delta=n}} \left( Y - t_\Gamma\left( \frac{\alpha \tau+\beta}{\delta} \right)\right)\text{.}
  \end{align}
\end{Definition}

The modular equations may be expressed as a polynomials with integer coefficients in the
parameter $Y$ and $t_{\Gamma}$ \cite[Proposition 2.5]{MR1400423}.

\begin{Theorem} For each normalized Hauptmodul  $t_\Gamma$ and $\gcd(n,N) = 1$,
 $$\Psi_{n}^{t_{\Gamma}}(Y) \in \Bbb Z [ t_\Gamma ,Y].$$
\end{Theorem}

For each $\tau$ in Table \ref{tab_matrix}, we demonstrate that $X(\tau)$ is a root of a modular
equation
$\Psi_{n}^{X}(X)$ for some $n$ relatively prime to $20$.

\begin{Theorem}
  For $n\ge 2 $ with $\gcd(n,20) = 1$, and for each $\tau \in \Bbb H$ defined in Table \ref{tab_matrix}, we have $\Psi_{n}^{X}(X) = 0$,
  where $X = X(\tau)$.
\end{Theorem}

\begin{proof}
  To show each $\tau$ in Table \ref{tab_matrix} satisfies $\Psi_{n}^{X}(X) = 0$,
  where $X = X(\tau)$, we determine an integer $n$ and pairwise relatively prime integers
$\alpha, \beta, \delta$ such that $\alpha \delta = n$ and $0\le \beta
<\delta$ with  $(\alpha \tau + \beta)/\delta$ equal to the image of
$\tau$ under an Atkin-Lehner involution from $W_{e}$
for some $e\mid\mid
20$. Since $X(\tau)$ is invariant under each Atkin-Lehner involution,
we see that $X(\tau) - X((\alpha \tau + \beta)/\delta) = 0$, so
$X = X(\tau)$ satisfies $\Psi_{n}^{X}(X) = 0$.
\end{proof}

To evaluate $X(\tau)$ for each quadratic irrational $\tau$ listed in
Table \ref{tab_matrix}, a sufficiently accurate approximation of
$X(\tau)$ may be used to distinguish the corresponding root of
$\Psi_{n}^{X}(X)$ and to evaluate $X(\tau)$. We illustrate the
technique for deriving expansions for $1/\pi$ in the next
example. The procedure relies on computing modular equations as in
\cite{MR2073912}.

Let $\tau_{0} = \tau(20,-40,23) = \frac{1}{10} \left(10+i
  \sqrt{15}\right)$. Then
\begin{align}
  \label{eq:13}
  \frac{\tau_{0} + 2}{3}  = \frac{20\tau_{0} -21}{20\tau_{0} -20}
\end{align}
Hence, with $(a,b,c,d) = (1,-21,1,-1)$ and $e=N = 20$, we have 
M:= \begin{align}
  \label{eq:5}
  \begin{pmatrix}
    ea & b \\ Nc & ed
  \end{pmatrix}
=
  \begin{pmatrix}
20 & -21 \\ 20 & -20 
\end{pmatrix} \in W_{20}.
\end{align}
Therefore, 
\begin{align}
  \label{eq:15}
  X \left ( \frac{\tau_{0} + 2}{3} \right ) = X(\tau_{0}), 
\end{align}
and so $X = X(\tau)$ is a root of the modular equation
\begin{align}
  \Psi_{3}(X,Y) = X^4-256 X^3& Y^3+192 X^3 Y^2-30 X^3 Y+192 X^2 Y^3-93 X^2
  Y^2 \nonumber \\ &+12 X^2 Y-30 X Y^3+12 X Y^2-X Y+Y^4 = 0. \label{me}
\end{align}
Hence, $X(\tau)$ is a root of
 $$\Psi_{3}(X,X) =-X^{2} (X-1)  (16 X-1) \left(16 X^2-7 X+1\right) = 0.$$
By numerically
approximating $X(\tau)$ and applying \eqref{me}, we deduce 
\eqref{me}, 
\begin{eqnarray}
  \label{eq:16}
  X(\tau_{0}) = 1/16,\qquad \frac{dY}{dX} = -1, \qquad  \frac{d^{2}Y}{dX^{2}} \bigg |_{\tau = \tau_{0}} =  - \frac{32}{3}.
\end{eqnarray}
Moreover, we may apply \eqref{me} to expand $Y$ about $X(\tau_{0})$
\begin{align} \label{expand}
  X(M\tau) = \sum_{k=0}^{\infty} Y^{(k)}(X(\tau_{0}))
   \frac{( X(\tau) - X(\tau_{0}))^{k}}{k!},\qquad Y^{(k)} = \frac{dY^{k}}{dX^{k}}. 
\end{align}
By applying $\frac{1}{2\pi i} \frac{d}{d\tau}$ twice to both sides of
\eqref{expand} and
applying \eqref{dX}, 
\begin{align}
  \label{eq:18}
  & \frac{20i}{\pi W\left ( 1 - \frac{dY}{dX} \right ) (20 \tau
  -20)} \bigg |_{\tau = \tau_{0}} \\  &= \left ( X \frac{dZ}{dX} + \left ( 1 +
  \frac{X}{W} \frac{dW}{dX} + X \frac{\frac{d^{2}Y}{dX^{2}}}{\frac{dY}{dX} \left ( 1 - \frac{dY}{dX} \right )} \right ) Z
  \right ) \bigg |_{\tau = \tau_{0}}, \nonumber
\end{align}
where 
\begin{align}
  \label{eq:17}
 W := W(X) :=\sqrt{(1-4X)(1-12X+16X^{2})}.
\end{align}
By \eqref{eq:16}, \eqref{eq:18}, and the fact that $Z = \sum_{n=0}^{\infty} A_{n}X^{n}$, we obtain
\begin{align}
  \label{eq:19}
  \frac{1}{\pi} = \frac{3}{8}\sum_{n=0}^{\infty} \frac{A_{n}}{16^{n}}
  \left ( n + \frac{1}{6} \right ).
\end{align}
Similarly, for each $\tau_{0}$ such that $X(\tau_{0})$ lies in the
radius of convergence for Theorem \ref{ZXeq}, we use the modular
equation to deduce the series expansions for $\pi$. The matrix $M$ under which $X$ is invariant, and
other parameters for each modular equation appear in Table
\ref{tab_matrix}. 
We restrict ourselves to singular values of degree no
more than $2$ over $\Bbb Q$. 

In order to formulate a complete list of algebraic $\tau$ such that
$X(\tau)$ has algebraic degree two over $\Bbb Q$, e use well known facts about the $j$ invariant
  \cite{MR1513075}. First, for algebraic $\tau$, the only algebraic values of $j(\tau)$
  occur at $\Im \tau>0$ satisfying $a\tau^{2}
  + b\tau + c = 0$ for $a,b,c \in \Bbb Z$, with $d =
  b^{2} - 4ac <0$. Moreover, $$[\Bbb Q(j(\tau)):\Bbb Q] = h(d),$$ where
  $h(d)$ is the class number. Since there is a polynomial relation
  $P(X,j)$ between $X$ and $j$ of degree $2$ \cite[Remark
  1.5.3]{MR1400423}, we have $$[\Bbb Q(j(\tau)):\Bbb Q] \le  2 [\Bbb
  Q(X(\tau)):\Bbb Q],$$ and so values $\tau$ with  $[\Bbb Q(X(\tau )): \Bbb Q]
  \le 2$ satisfy $$[\Bbb Q(j(\tau)):\Bbb Q] =h(d) \le 4.$$ Therefore, the
  bound $|d| \le 1555$ for $h(d) \le 4$ from \cite{MR2031415}
  implies that Algorithm \ref{br} gives a complete list of algebraic
  $(\tau, X(\tau))$ with $[\Bbb Q(X(\tau )): \Bbb Q] \le 2$: \\

\begin{Algorithm} \label{br}
  \indent For each discriminant $-1555 \le d \le -1$, 
  \begin{enumerate}
  \item List all primitive reduced $\tau = \tau(a,b,c)$ of
    discriminant $d$ in a fundamental domain for $PSL_{2}(\Bbb Z)$. Translate these values via a set of coset representatives for $\Gamma_{0}(20)$ to a fundamental domain for $\Gamma_{0}(20)$. 
\item Factor the resultant of $P(X,Y)$ and the class polynomial
  $$H_{d}(Y) = \prod_{\substack{(a,b,c)\text{ reduced, primitive} \\
      d = b^{2}-4ac}} \left ( Y - j\Bigl (
  \frac{-b+\sqrt{d}}{2 a} \Bigr ) \right ).$$  Linear and quadratic factors
  of the resultant correspond to a complete list of $ X(\tau)$, for
  $\tau$ of discriminant $d$, such that $[\Bbb Q(X(\tau )): \Bbb Q]
  \le 2$.
\item Associate candidate values $\tau$ from Step 1 to $X$ by
  approximating $X(\tau)$. For each approximation, prove the evaluation $X=X(\tau)$ by deriving a
  corresponding modular equation for which $X(\tau)$ are
  roots. The proof of each evaluation may be accomplished through a
  rigorous derivation of the first decimal digits of $X(\tau)$ and a
  comparison of these values with those of the roots of the modular equation. 
  \end{enumerate}
  \end{Algorithm}

Algorithm \ref{br} was implemented in Mathematica resulting in Table
\ref{tab_x_values_17A}. 

\begin{table}
\begin{equation*}
\begin{array}{lll}
b^2-4ac & \tau(a,b,c) & X(\tau)  \\ 
 -16 & (4,-8,5) & \frac{1}{8}(7 - 3 \sqrt{5}) \\
 \left [{-40 \atop -160} \right ]& \left [ {(10,-20,11)  \atop
                                   (40,-80,41)} \right ]& -3/2+ \sqrt{5/2} \\
 -240 & (20,-40,23) & 1/16 \\
 -256 & (20,-12,5) & i/8\sqrt{2}\\
 -256 & (20,-28,13) & -i/8\sqrt{2} \\
 -320 & (20,-20,9) & \frac{1}{16}(1 - \sqrt{5}) \\
 -400 & (100,-200,101) & \frac{1}{8}(7 - 3\sqrt{5}) \\
 -480 & (20,-20,11) & \frac{1}{8}(-7 + 3\sqrt{5}) \\
 -640 & (20,-20,13) & \frac{1}{8}(3 - \sqrt{10}) \\
 -880 & (20,-40,31) & \frac{1}{32}(7 - 3\sqrt{5})\\
 -960 & (20,-20,17) & \frac{1}{16}(-4+\sqrt{15}) \\
 -1120 & (20,-20,19) & \frac{1}{8}(-47+21\sqrt{5}) \\
 -1360 & (20,-40,37) & 1/18 \sqrt{85}+166 \\
 -2080 & (20,20,31) & -161/4 +18 \sqrt{5} \\
-3040 & (20,20,43) & -721/4+57 \sqrt{10}
\end{array}
\end{equation*}
\caption{Complete list of quadratic singular values of $X(\tau)$ within the radius of convergence of Theorem \ref{ZXeq}.}
\label{tab_x_values_17A}
\end{table}

In the final two
tables, we list 
constants defining level $20$ expansions 
\begin{align}
  \label{fm}
 \frac{1}{\pi} =
 A\sum_{n=0}^{\infty} a_{n}  (n + B )C^{n}.
\end{align}
\begin{table}
\begin{equation*}
\begin{array}{clllc}
\tau(a,b,c) & n & e & (\alpha, \beta; 0, \delta) & \gamma \in W_{e} \\ 
  (4,-8,5) & 29 & 4 & (1, 5; 0, 29) & (4,-5;20,-24) \\
(10,-20,11)  & 13 &5 & (1, 9; 0, 13) & (5,-5;20,-30) \\
 (40,-80,41) & 13 & 5 & (1,4;0,13) & (15,-17;40,-45) \\
(20,-40,23) & 3 & 20& (1, 2; 0, 3)    & (20,-21,20,-20)\\
 (20,-12,5) & 17 & 4 &  (1,10,0,17) & (12,-5; 20,-8) \\
(20,-28,13) & 17 & 4 & (1, 13; 0, 17) & (16,-13; 20, -16) \\
 (20,-20,9) & 9 & 20 & (1, 0; 0, 9) & (0,-1; 20,-20) \\
 (100,-200,101) &29 & 4& (1, 6;0, 29) & (24,-25; 100, -104)\\
 (20,-20,11) & 11 & 20 & (1, 0; 0, 11) & (0, -1; 20, -20)   \\
 (20,-20,13) & 13 & 20 & (1, 0; 0, 13) & (0, -1; 20, -20)  \\
(20,-40,31) & 11 & 20 & (1,10; 0, 11) & (20, -21; 20, -20) \\
 (20,-20,17) & 17 & 20 &(1, 0; 0, 17) & (0, -1; 20, -20)  \\
 (20,-20,19) & 19 & 20 &  (1, 0; 0, 19) & (0, -1; 20, -20) \\
 (20,-40,37) & 17 & 20 & (1, 16; 0, 17) & (0, -21; 20, -20) \\
(20,20,31) & 31 & 20 & (1,1;0,31) & (0, -1; 20, -20) \\
(20,20,43) & 43 & 20&  (1,0;0, 43) & (0, -1; 20, -20) 
\end{array}
\end{equation*}
\caption{Matrices $(\alpha, \beta; 0, \delta)$ mapping $\tau$ to
  the image under an element from $W_{e}$, demonstrating $X(\tau)$ is a root of $\Psi_{n}^{X}(X)$}
\label{tab_matrix}
\end{table}

%
%

\begin{table}
\begin{equation*}
\begin{array}{lll} 
 \ \ \ \ \ \ \ \ \ \  A  & \ \ \ \ \ \ \  B & \ \ \ \ \ \ \   C\\  
2 \sqrt{9 \sqrt{5}-20} & \frac{1}{4}
  (3-\sqrt{5}) & \frac{1}{8} (7-3
  \sqrt{5}) \\  \sqrt{506-160 \sqrt{10}} & \frac{1}{6}
                                                    \left(4-\sqrt{10}\right)
            & -3/2+ \sqrt{5/2}  \\ 
 3/8 & 1/6 & 1/16 \\ 
 \frac{1}{5}
  \sqrt{8-\frac{31 i}{\sqrt{2}}} & \frac{1}{66} \left(31-8 i
                                    \sqrt{2}\right) & i/8
                                                       \sqrt{2} \\
  \frac{1}{5}
  \sqrt{8+\frac{31 i}{\sqrt{2}}} & \frac{1}{66} \left(31+8 i
                                    \sqrt{2}\right) & -i/8
                                                       \sqrt{2} \\
   \sqrt{\frac{1}{2} \left(\sqrt{5}+2\right)} & 1/2 &
                                                                \frac{1}{16}
                                                                \left(1-\sqrt{5}\right)
  \\  
  \sqrt{\sqrt{10}-1} & \frac{1}{6} \left(5-\sqrt{10}\right) &
                                                                  \frac{1}{8}
                                                                  \left(3-\sqrt{10}\right)
  \\  11/8 & \frac{1}{22} \left(13-4 \sqrt{5}\right) &
                                                                 \frac{1}{32}
                                                                 \left(7-3
                                                                 \sqrt{5}\right)
  \\  \frac{1}{4} \sqrt{636-153 \sqrt{15}} & \frac{1}{462}
                                                     \left(261-40
                                                     \sqrt{15}\right)
            & \frac{1}{16} \left(\sqrt{15}-4\right) \\  14 \sqrt{805-360 \sqrt{5}}
      & \frac{1}{140} \left(105-34 \sqrt{5}\right) & \frac{1}{8}
                                                       \left(21
                                                       \sqrt{5}-47\right)
  \\  12 \sqrt{697 \sqrt{85}-6426} & \frac{1}{204}
                                             \left(153-13
                                             \sqrt{85}\right) &
                                                                 \frac{1}{8}
                                                                 \left(83-9
                                                                 \sqrt{85}\right)
  \\  78 \sqrt{14445-6460 \sqrt{5}} & \frac{1}{780}
                                              \left(585-212
                                              \sqrt{5}\right) &
                                                                 \frac{1}{4}
                                                                 \left(72
                                                                 \sqrt{5}-161\right)
  \\  646 \sqrt{27379-8658 \sqrt{10}} & \frac{969-259 \sqrt{10}}{1292} & \frac{1}{4} \left(228 \sqrt{10}-721\right) \\ 
\end{array}
\end{equation*}
\caption{Constants defining level $20$ expansions \eqref{fm}. }
\label{tab_x_values_17}
\end{table}

\clearpage

\bibliographystyle{plain}
\bibliography{references}

\end{document}